\documentclass[draft]{amsart}
\usepackage{amsmath}
\usepackage{amssymb}
\usepackage{latexsym}
\usepackage{graphicx}
\usepackage{epstopdf}
\usepackage[final]{hyperref}
\hypersetup{colorlinks=true, linkcolor=blue, anchorcolor=blue, citecolor=red, filecolor=blue, menucolor=blue, pagecolor=blue, urlcolor=blue}

\newcommand{\abs}[1]{\left\vert #1 \right\vert}

\newcommand{\norm}[1]{\left\Vert #1 \right\Vert}

\newcommand{\bignorm}[1]{\bigl\Vert #1 \bigr\Vert}

\newcommand{\C}{\mathbb{C}}

\newcommand{\R}{\mathbb{R}}

\newcommand{\angles}[1]{\langle #1 \rangle}

\DeclareMathOperator{\diag}{diag}

\DeclareMathOperator{\im}{Im}
\DeclareMathOperator{\re}{Re}

\newtheorem{theorem}{Theorem}[section]

\newtheorem{lemma}{Lemma}[section]

\theoremstyle{definition}

\theoremstyle{remark}

\numberwithin{equation}{section}
\title[Chern-Simons-Higgs with finite energy]{Global well-posedness of the Chern-Simons-Higgs equations with finite energy}

\author[S. Selberg]{Sigmund Selberg}
\address{Department of Mathematical Sciences\\
Norwegian University of Science and Technology\\
Alfred Getz' vei 1\\
N-7491 Trondheim\\ Norway}
\email{sigmund.selberg@math.ntnu.no}

\author[A. Tesfahun]{Achenef Tesfahun}
\address{Department of Mathematical Sciences\\
Norwegian University of Science and Technology\\
Alfred Getz' vei 1\\
N-7491 Trondheim\\ Norway}
\email{achenef.tesfahun@math.ntnu.no}


\begin{document}

\begin{abstract}
We prove that the Cauchy problem for the Chern-Simons-Higgs equations on the (2+1)-dimensional Minkowski space-time is globally well posed for initial data with finite energy. This improves a result of Chae and Choe, who proved global well-posedness for more regular data. Moreover, we prove local well-posedness even below the energy regularity, using the the null structure of the system in Lorenz gauge and bilinear space-time estimates for wave-Sobolev norms.
\end{abstract}

\maketitle

\section{Introduction}\label{Introduction}

The (2+1)-dimensional abelian Chern-Simons-Higgs model was proposed by Hong, Kim and Pac \cite{hkp1990} and Jackiw and Weinberg~\cite{jw1990} in the study of vortex solutions in the abelian Chern-Simons theory. The Lagrangian for the model is
$$
  \mathcal L = \frac{\kappa}{4} \epsilon^{\mu\nu\rho} A_\mu F_{\nu\rho} + D_\mu\phi \overline{D^\mu\phi} - V\left(\abs{\phi}^2\right),
$$
on the Minkowski space-time $\R^{1+2} = \R_t \times \R_x^2$ with metric $g_{\mu\nu} = \diag(1,-1,-1)$. Here $D_\mu = \partial_\mu - iA_\mu$ is the covariant derivative associated to the gauge field $A_\mu \in \R$, $F_{\mu\nu} = \partial_\mu A_\nu - \partial_\nu A_\mu$ is the curvature, $\phi \in \C$ is the Higgs field, $V(\abs{\phi}^2) \in \R$ is a Higgs potential, $\kappa > 0$ is a Chern-Simons coupling constant, and $\epsilon^{\mu\nu\rho}$ is the skew-symmetric tensor with $\epsilon^{012}=1$. Greek indices range from 0 to 2, Latin indices from 1 to 2, and repeated upper/lower indices are implicitly summed.

The corresponding Euler-Lagrange equations are
\begin{equation}\label{CSH}
  F_{\mu\nu} = \frac{1}{\kappa} \epsilon_{\mu\nu\rho} J^\rho,
  \qquad
  D_\mu D^\mu \phi = -\phi V'\left( \abs{\phi}^2 \right),
\end{equation}
where
$$
  J^\rho = 2\im\left(\overline{\phi}D^\rho\phi\right).
$$
There is a conserved energy,
$$
  E(t) = \int_{\R^2} \left( \sum_{\mu=0}^2 \abs{D_\mu \phi(t,x)}^2 + V\left(\abs{\phi(t,x)}^2\right) \right) \, dx,
$$
and the equations are invariant under the gauge transformations
\begin{equation}\label{GaugeTransformation}
  A_\mu \to A_\mu' = A_\mu + \partial_\mu\chi,
  \qquad
  \phi \to \phi' = e^{i\chi} \phi,
  \qquad
  D_\mu \to D_\mu' = \partial_\mu - iA_\mu',
\end{equation}
hence we may impose an additional gauge condition. In this paper we rely on the Lorenz condition $\partial^\mu A_\mu = 0$.

A typical potential is $V(r) = \kappa^{-2} r(1-r)^2$ (see \cite{hkp1990, jw1990}), in which case there are two possible boundary conditions to make the energy finite: Either $\abs{\phi} \to 1$ as $\abs{x} \to \infty$ (the \emph{topological} case) or $\abs{\phi} \to 0$ as $\abs{x} \to \infty$ (the \emph{non-topological} case).

We are interested in the Cauchy problem for the non-topological case, which received considerable attention recently.  Local well-posedness for low-regularity data was studied in \cite{h2007,b2009,jy2011,h2011}, but the energy regularity was not quite reached; Huh \cite{h2011} came arbitrarily close to energy using the Coulomb gauge. In this paper we close the remaining gap, using the Lorenz gauge. In fact, we prove that the problem is locally well posed not only at the energy regularity but even a little below it. From the local finite-energy well-posedness we get the corresponding global result by exploiting the conservation of energy and the residual gauge freedom within Lorenz gauge. In particular, we improve the earlier result of Chae and Choe \cite{cc2002}, who proved global well-posedness for more regular data, namely with one derivative extra in $L^2$ compared with energy.

In order to pose the Cauchy problem one should know the observables $F_{\mu\nu}$, $J^\rho$ and $E$ at time $t=0$, so it suffices to specify $\phi(0)$ and $D_\mu\phi(0)$. Since we are interested in the non-topological case we assume $V(0) = 0$. Moreover we assume that $V'(r)$ has polynomial growth, hence $E(0)$ is absolutely convergent if
\begin{gather}
  \label{FE1}
  D_\mu\phi(0) \in L^2,
  \\
  \label{FE2}
  \phi(0) \in L^p \quad \text{for all $2 \le p < \infty$},
\end{gather}
which imply
\begin{equation}\label{CurrentRegularity}
  J^\rho(0) = 2\im\bigl(\overline{\phi(0)} D^\rho\phi(0)\bigr) \in \dot H^{-1/2},
\end{equation}
since by the Hardy-Littlewood-Sobolev inequality on $\R^2$,
\begin{equation}\label{Current}
  \norm{2\im(\overline{f}g)}_{\dot H^{-1/2}} \le C \norm{\overline{f}g}_{L^{4/3}} \le \norm{f}_{L^4} \norm{g}_{L^2}.
\end{equation}

Here $\dot H^s = \dot H^s(\R^2)$, $\abs{s} < 1$, is the completion of $\mathcal S(\R^2)$ with respect to the norm $\norm{f}_{\dot H^s} = \bignorm{\abs{\xi}^s \widehat f(\xi)}_{L^2}$, where $\widehat f(\xi)$ is the Fourier transform of $f(x)$. A direct characterization is
$$
  \dot H^s = \mathcal F^{-1}\bigl( L^2(\abs{\xi}^{2s} d\xi)\bigr) \qquad (\abs{s} < 1).
$$
Here $s > -1$ ensures that $\mathcal S \subset L^2(\abs{\xi}^{2s} d\xi)$ (densely), and $s < 1$ ensures that functions in $L^2(\abs{\xi}^{2s} d\xi)$ are tempered, so the inverse Fourier transform can be applied. We also need the inhomogeneous space $H^s=H^s(\R^2)$, which is the completion of $\mathcal S(\R^2)$ with respect to $\norm{f}_{H^s} = \bignorm{\angles{\xi}^s \widehat f(\xi)}_{L^2}$, where $\angles{\xi} = (1+\abs{\xi}^2)^{1/2}$.

Recall the Hardy-Littlewood-Sobolev inequalities
\begin{gather}
  \norm{f}_{L^q} \le C \norm{\abs{\nabla}^s f}_{L^p} \qquad \left(1 < p < q < \infty, \quad \frac{1}{p} - \frac{1}{q} = \frac{s}{2}\right),
  \\
  \norm{f}_{L^\infty} \le C \norm{\angles{\nabla}^s f}_{L^p} \qquad \left( p \ge 1, \quad s > \frac{2}{p}\right),
\end{gather}
where $\abs{\nabla} = (-\Delta)^{1/2}$ and $\angles{\nabla} = (1-\Delta)^{1/2}$. In particular, $H^1 \subset L^p$ for all $2 \le p < \infty$, and
\begin{equation}\label{CovariantEstimate}
  \norm{fg}_{L^2} \le \norm{f}_{L^4} \norm{g}_{L^4} \lesssim \norm{f}_{\dot H^{1/2}} \norm{g}_{H^1}.
\end{equation}
The notation $a \lesssim b$ stands for $a \le Cb$.

\section{Main results}

Since the value of the positive constant $\kappa$ is irrelevant for our analysis, we shall set $\kappa = 1$. Augmented with the Lorenz gauge condition $\partial^\mu A_\mu = 0$, \eqref{CSH} reads
\begin{subequations}\label{CSHL}
\begin{align}
  \label{CSHLa}
  \partial_t A_j - \partial_j A_0 &= \epsilon_{jk} J^k,
  \\
  \label{CSHLb}
  \partial_1 A_2 - \partial_2 A_1 &= J_0,
  \\
  \label{CSHLc}
  \partial_1 A_1 + \partial_2 A_2 &= \partial_t A_0,
  \\
  \label{CSHLd}
  D_\mu D^\mu \phi &= -\phi V'\left( \abs{\phi}^2 \right),
\end{align}
\end{subequations}
where $J^\rho = 2\im\left(\overline{\phi}D^\rho\phi\right)$ and $\epsilon_{jk}$ is the skew-symmetric tensor with $\epsilon_{12}=1$.

We pose the Cauchy problem in terms of data for $A_\mu$ and $(\phi,\partial_t \phi)$. The question then arises: What are the natural data spaces, given that \eqref{FE1}--\eqref{CurrentRegularity} should hold? To answer this, first note that the Lorenz condition leaves some gauge freedom, since it is preserved by \eqref{GaugeTransformation} if $\square \chi = 0$, where $\square = \partial^\mu\partial_\mu = \partial_t^2 - \Delta$ is the d'Alembertian. So formally, at least, we may impose the initial constraints
\begin{equation}\label{EllipticGauge}
  A_0(0) = 0, \qquad \partial^j A_j(0) = 0,
\end{equation}
for if these are not already satisfied, they will be after a gauge transformation \eqref{GaugeTransformation} with gauge function $\chi$ satisfying
\begin{equation}\label{GaugeFunction}
  \square \chi = 0, \qquad \Delta \chi(0) = \partial^j A_j(0), \qquad \partial_t \chi(0) = -A_0(0).
\end{equation}
But from \eqref{EllipticGauge} and \eqref{CSHLb} we get
\begin{equation}\label{EllipticPotential}
  \Delta A_j(0) = \epsilon_{jk} \partial^k J_0(0),
\end{equation}
so $A_j(0)$ should be in $\dot H^{1/2}$, recalling \eqref{CurrentRegularity}. Then from \eqref{FE1} and \eqref{FE2} we infer that $(\phi,\partial_t\phi)(0) \in H^1 \times L^2$, since $\partial_t\phi(0) = D_0\phi(0)$ and $\partial_j\phi(0) = D_j\phi(0) + iA_j(0)\phi(0)$, and the last term is in $L^2$ by \eqref{CovariantEstimate}.

So now we know what the correct data spaces for $A_\mu$ and $(\phi,\partial_t \phi)$ are. Note, however, that \eqref{CSHLb} imposes an initial constraint. The following lemma shows that given any data for $(\phi,\partial_t\phi)$ in $H^1 \times L^2$, there exists an initial potential $A_\mu(0)$ satisfying this constraint as well as the finite energy requirements \eqref{FE1} and \eqref{FE2}.

\begin{lemma}
Given data
$$
  (\phi,\partial_t\phi)(0) \in H^{1} \times L^2, 
$$
there exists an initial potential
$$
  A_\mu(0) \in \dot H^{1/2}
$$
satisfying \eqref{EllipticGauge}, and \eqref{CSHLb} at $t=0$. Moreover, \eqref{FE1} and \eqref{FE2} are satisfied.
\end{lemma}

\begin{proof} First note that \eqref{FE2} holds by the embedding $H^1 \subset L^p$, $2 \le p < \infty$. Set $A_0(0) = 0$ and
$$
  A_j(0) = - (-\Delta)^{-1/2} \epsilon_{jk} R^k J_0(0),
$$
where $R_k = (-\Delta)^{-1/2}\partial_k$ is the Riesz transform. By \eqref{Current},
$$
  \norm{A_j(0)}_{\dot H^{1/2}}
  \lesssim \norm{J_0(0)}_{\dot H^{-1/2}}
  \lesssim \norm{\phi(0)}_{L^4} \norm{\partial_t\phi(0)}_{L^2}
  \lesssim \norm{\phi(0)}_{H^1} \norm{\partial_t\phi(0)}_{L^2},
$$
and $D_\mu\phi(0) \in L^2$ follows from \eqref{CovariantEstimate}. By \eqref{EllipticPotential}, $\Delta(\partial_1 A_2(0) - \partial_2 A_1(0) - J_0(0)) = 0$ and $\Delta(\partial_1 A_1(0) + \partial_2 A_2(0)) = 0$, and in general, $\Delta f=0$ implies $f=0$ if $f \in \dot H^{-1/2}$, since $\widehat f$ is a tempered function. Thus \eqref{EllipticGauge} holds, as does \eqref{CSHLb} at $t=0$.
\end{proof}

More generally, we shall prove local well-posedness for any data
\begin{equation}\label{CSHLdata}
  A_\mu(0) \in \dot H^{1/2}, \qquad (\phi,\partial_t\phi)(0) \in H^{1} \times L^2,
\end{equation}
satisfying \eqref{CSHLb} initially:
\begin{equation}\label{CSHLconstraint}
  \partial_1 A_2(0) - \partial_2 A_1(0) = J_0(0) = 2\im\left(\overline{\phi(0)}D_0\phi(0)\right).
\end{equation}
Then $D_\mu\phi(0) = \partial_\mu\phi(0) - i A_\mu(0)\phi(0)$ is in $L^2$, with norm bounded in terms of the norm of \eqref{CSHLdata}, in view of \eqref{CovariantEstimate}.

\begin{theorem}\label{Thm1}
The Chern-Simons-Higgs-Lorenz Cauchy problem \eqref{CSHL}, \eqref{CSHLdata}, \eqref{CSHLconstraint} is locally well posed, for any potential $V \in C^\infty(\R_+;\R)$ such that $V(0) = 0$ and all derivatives of $V$ have polynomial growth. More precisely, there exists a time $T > 0$, which is a decreasing and continuous function of the data norm
$$
  \sum_{\mu=0}^2\norm{A_\mu(0)}_{\dot H^{1/2}} + \norm{\phi(0)}_{H^1} + \norm{\partial_t \phi(0)}_{L^2},
$$
and a solution $(A,\phi)$ of \eqref{CSHL} on $(-T,T) \times \R^2$ with the regularity
\begin{equation}\label{SolutionSpace}
\begin{gathered}
  A_\mu \in C([-T,T];\dot H^{1/2}),
  \\
  \phi \in C([-T,T];H^1), \qquad \partial_t \phi \in C([-T,T];L^2).
\end{gathered}
\end{equation}
The solution is unique in a certain subset of this regularity class. Moreover, the solution depends continuously on the data, and higher regularity persists. In particular, if the data are smooth, then so is the solution.
\end{theorem}

The proof is given in Section \ref{EnergyLWP}.

Our plan is now to show that (i) the time $T$ in fact only depends on $\mathcal I(0)$, where
$$
  \mathcal I(t) = \norm{\phi(t)}_{L^2} + \sum_{\mu=0}^2 \norm{D_\mu\phi(t)}_{L^2},
$$
and (ii) $\mathcal I(t)$ is a priori controlled for all time in terms of $E(0)$ and $\norm{\phi(0)}_{L^2}$. Then it will of course follow that the solutions extend globally in time.

To prove (i) we apply the gauge transformation \eqref{GaugeTransformation} with $\chi$ satisfying \eqref{GaugeFunction}.

\begin{lemma}\label{GaugeLemma}
Given data $A_\mu(0) \in \dot H^{1/2}$, there exists $\chi(t,x)$ with the regularity
$$
  \chi \in C(\R^{1+2}), \qquad \partial_\mu \chi \in C(\R;\dot H^{1/2}),
$$
and satisfying \eqref{GaugeFunction}.
\end{lemma}

\begin{proof} The solution of \eqref{GaugeFunction} is
$$
  \chi(t) = \cos(t\abs{\nabla}) f + \sin(t\abs{\nabla})\abs{\nabla}^{-1} g,
$$
where $g = - A_0(0) \in \dot H^{1/2}$ and $f$ should satisfy
\begin{equation}\label{chiElliptic}
  \Delta f = \partial^j A_j(0).
\end{equation}

First, if the Fourier transform of $A_\mu(0)$ is supported in $\{ \xi \in \R^2 \colon \abs{\xi} \ge 1\}$, then $g \in H^{1/2}$, and \eqref{chiElliptic} has a unique solution $f \in H^{3/2}$, so $\chi \in C(\R;H^{3/2}) \subset C(\R^{1+2})$, and $\partial_\mu \chi \in C(\R;H^{1/2})$.

Now assume that $A_\mu(0)$ has Fourier support in $\{ \xi \in \R^2 \colon \abs{\xi} < 1\}$. Then $A_\mu(0)$ is smooth, but it is not obvious that \eqref{chiElliptic} has a solution (what is clear is that the solution, if it exists, will be smooth). Formally, $f$ should be given by, with $R_j = (-\Delta)^{-1/2} \partial_j$ the Riesz transform,
$$
  f = - (-\Delta)^{-1/2} R^j A_j(0),
$$
but it is not clear that this is meaningful. However, if we take the gradient we get something well-defined:
$$
  \partial_k f = F_k \equiv - R_k R^j A_j(0) \in \dot H^{1/2} \cap C^\infty.
$$
But $(F_1,F_2)$ is a smooth vector field on $\R^2$ with zero curl:
$$
  \partial_1 F_2 - \partial_2 F_1 = 0,
$$
hence $(F_1,F_2)$ is the gradient of a smooth function, which we denote $f$. Then it follows that \eqref{chiElliptic} is satisfied. So now $f,g \in C^\infty(\R^2)$, hence $\chi \in C^\infty(\R^{1+2})$. Moreover, $\partial_j f, g \in \dot H^{1/2}$, so $\partial_\mu \chi \in C(\R;\dot H^{1/2})$.
\end{proof}

We also need the covariant Sobolev inequality, proved in \cite{gv1981},
\begin{equation}\label{gv}
  \norm{\phi(0)}_{L^p} \le C\norm{\phi(0)}_{L^2}^{2/p}\left(\sum_{j=1}^2 \norm{D_j \phi(0)}_{L^2}\right)^{1-2/p}
  \qquad (2 < p < \infty),
\end{equation}
which holds for all $\phi(0) \in H^1$ such that $D_j\phi(0) \in L^2$ (the regularity of the real-valued functions $A_j(0)$ is irrelevant here).

\begin{theorem}\label{Thm2}
The solution $(A,\phi)$ from Theorem \ref{Thm1} exists up to a time $T > 0$ which is a continuous and decreasing function of
$$
  \mathcal I(0) = \norm{\phi(0)}_{L^2} + \sum_{\mu=0}^2 \norm{D_\mu\phi(0)}_{L^2}.
$$
\end{theorem}

\begin{proof} Given data \eqref{CSHLdata} satisfying \eqref{CSHLconstraint}, apply the gauge transformation \eqref{GaugeTransformation} with $\chi$ as in Lemma \ref{GaugeLemma}. Then \eqref{GaugeTransformation} preserves the regularity \eqref{SolutionSpace}, as does its inverse, obtained by replacing $\chi$ by $-\chi$. In the new gauge,
$$
  A_0'(0) = 0, \qquad \partial^j A_j'(0) = 0,
$$
and by the latter combined with \eqref{CSHLconstraint} (which is gauge invariant),
$$
  \Delta A_j'(0) = \epsilon_{jk} \partial^k J_0(0).
$$
Since we know that $A_j'(0)$ belongs to $\dot H^{1/2}$, and since in general $\Delta f = 0$ implies $f=0$ if $f \in \dot H^{1/2}$ (then $\widehat f$ is a tempered function), we conclude that
$$
  A_j'(0) = - (-\Delta)^{-1/2} \epsilon_{jk} R^k J_0(0),
$$
where $R_k = (-\Delta)^{-1/2}\partial_k$ is the Riesz transform. Thus, by \eqref{Current},
\begin{equation}\label{Aprime}
  \norm{A_j'(0)}_{\dot H^{1/2}}
  \lesssim \norm{J_0(0)}_{\dot H^{-1/2}}
  \lesssim \norm{\phi(0)}_{L^4} \norm{D_0\phi(0)}_{L^2} \lesssim \mathcal I(0)^2,
\end{equation}
where we applied \eqref{gv} in the last step. Moreover,
\begin{gather*}
  \phi'(0) = e^{i\chi(0)}\phi(0),
  \\
  \partial_\mu\phi'(0) = D_\mu'\phi'(0) + iA_\mu'(0)\phi'(0) = e^{i\chi(0)}D_\mu\phi(0) + ie^{i\chi(0)}A_\mu'(0)\phi(0),
\end{gather*}
hence
$$
  \norm{\phi'(0)}_{L^2} + \sum_{\mu=0}^2 \norm{\partial_\mu\phi'(0)}_{L^2}
  \le
  \mathcal I(0) + \sum_{\mu=0}^2 \norm{A_\mu'(0)}_{L^4} \norm{\phi(0)}_{L^4}
  \lesssim
  \mathcal I(0) + \mathcal I(0)^3,
$$
where we used \eqref{gv} and \eqref{Aprime}.

Thus, applying Theorem \ref{Thm1} we get the solution $(A',\phi')$ up to a time $T > 0$ which is a continuous and decreasing function of $\mathcal I(0)$. Finally, reverse the gauge transformation to get the solution $(A,\phi)$.
\end{proof}

Finally, we show that the solutions extend globally in time.

\begin{theorem}\label{Thm2b}
In addition to the hypotheses in Theorem \ref{Thm1}, assume that
$$
  V(r) \ge -\alpha^2 r
$$
for all $r \ge 0$ and some $\alpha > 0$. Then the solution $(A,\phi)$ from Theorem \ref{Thm1} exists globally in time and has the regularity \eqref{SolutionSpace} for all $T > 0$.
\end{theorem}

In view of Theorem \ref{Thm2} it suffices to show that
$$
  \mathcal I(t) = \norm{\phi(t)}_{L^2} + \sum_{\mu=0}^2 \norm{D_\mu\phi(t)}_{L^2}
$$
is a priori bounded on every finite time interval. For this, we rely of course on the conservation of energy (which is satisfied since our local solutions are limits of smooth solutions with compact spatial support). First we note, using $E(t)=E(0)$ and the assumption $V(r) \ge -\alpha^2 r$, that
\begin{equation}\label{Energy1}
  \norm{D_\mu\phi(t)}_{L^2}^2 = E(0) - \int V\left(\abs{\phi(t,x)}^2\right) \, dx
  \le \abs{E(0)} + \alpha^2 \norm{\phi(t)}_{L^2}^2.
\end{equation}
Then
\begin{align*}
  \frac{d}{dt} \left( \norm{\phi(t)}_{L^2}^2 \right)
  &= \int 2\re\left(\overline{\phi(t,x)} D_0\phi(t,x) \right) \, dx
  \\
  &\le 2 \norm{\phi(t)}_{L^2} \norm{D_0\phi(t)}_{L^2}
  \\
  &\le 2 \norm{\phi(t)}_{L^2} \left( \abs{E(0)} + \alpha^2 \norm{\phi(t)}_{L^2}^2  \right)^{1/2}
  \\
  &\le \alpha^{-1} \abs{E(0)} + 2 \alpha \norm{\phi(t)}_{L^2}^2,
\end{align*}
hence by Gr\"onwall's lemma,
\begin{equation}\label{Energy3}
  \norm{\phi(t)}_{L^2}^2 \le e^{2 \alpha \abs{t}} \left( \norm{\phi(0)}_{L^2}^2 + \abs{t} \alpha^{-1} \abs{E(0)} \right).
\end{equation}
By \eqref{Energy1} and \eqref{Energy3} we control $\mathcal I(t)$, and Theorem \ref{Thm2b} is proved.

It remains to prove Theorem \ref{Thm1}. Note that in Lorenz gauge, $\partial^\nu F_{\mu\nu} = -\square A_\mu$, so \eqref{CSHL} implies
\begin{equation}\label{CSHLwave}
  \square A_\mu = -\epsilon_{\mu\nu\rho} \partial^\nu J^\rho,
  \qquad
  \partial^\mu A_\mu = 0,
  \qquad
  D_\mu D^\mu\phi = - \phi V'\left( \abs{\phi}^2 \right),
\end{equation}
and this is the system we actually solve.

Then we have to check that, conversely, \eqref{CSHLwave} implies \eqref{CSHLa} and \eqref{CSHLb}, assuming that the latter two are satisfied at $t=0$. But then
$$
  v_j = \partial_t A_j - \partial_j A_0 - \epsilon_{jk} J^k,
  \qquad
  w = \partial_1 A_2 - \partial_2 A_1 - J_0,
$$
vanish at time $t=0$. Moreover, using \eqref{CSHLa}, $\square A_j = \epsilon_{jk}(-\partial_t J^k + \partial^k J_0)$, and $\partial^\mu J_\mu = 0$ (which follows from the last equation in \eqref{CSHLwave}), one finds
$$
  \partial_t v_j = \epsilon_{jk} \partial^k w,
  \qquad
  \partial_t w = \partial_1 v_2 - \partial_2 v_1,
$$
and these vanish at $t=0$ since $v_j$ and $w$ do. Taking another time derivative gives $\square v_j = \partial_j (\partial^k v_k)$ and $\square w = 0$. But $\square A_0 = - \partial_1 J_2 + \partial_2 J_1$ implies $\partial^k v_k=0$, hence
$$
  \square v_j = 0, \qquad \square w = 0.
$$
Since the data vanish, we conclude that $v_j=w=0$, so \eqref{CSHLa} and \eqref{CSHLb} hold.

Before proving Theorem \ref{Thm1}, we consider in the following section the problem of local well-posedness with minimal regularity, and it turns out that we can get below the energy regularity. Here we take data for $A_\mu$ in inhomogeneous Sobolev spaces.

\section{Low regularity local well-posedness}\label{LWP}

The system \eqref{CSHLwave} expands to
\begin{subequations}\label{Waves}
\begin{gather}
  \label{WaveA}
  (\square+1) A_\mu = - \epsilon_{\mu\nu\rho} \im Q^{\nu\rho}(\partial\overline{\phi},\partial\phi) + 2\epsilon_{\mu\nu\rho} \partial^\nu\left( A^\rho \abs{\phi}^2 \right) + A_\mu,
  \\
  \label{LG}
  \partial^\mu A_\mu = 0,
  \\
  \label{WaveB}
  (\square+1) \phi = 2i A_\mu \partial^\mu\phi + A_\mu A^\mu \phi - \phi V'\left( \abs{\phi}^2 \right) + \phi,
\end{gather}
\end{subequations}
where $Q_{\alpha\beta}(\partial u, \partial v) = \partial_\alpha u \partial_\beta v - \partial_\beta u \partial_\alpha v$ is the standard null form. Here we added $A_\mu$ and $\phi$ to each side of \eqref{WaveA} and \eqref{WaveB}, respectively, to get the operator $\square + 1$; this is done to avoid a singularity in \eqref{Splitting} below. We specify data
\begin{equation}\label{WaveData}
  A_\mu(0) \in H^s,
  \qquad (\phi,\partial_t\phi)(0) \in H^{s+1/2} \times H^{s-1/2}.
\end{equation}
The data for $\partial_t A_\mu$ are given by the constraints
\begin{align}
  \label{LorenzConstraint}
  \partial_t A_0(0) &= \partial_1 A_1(0) + \partial_2 A_2(0) \in H^{s-1},
  \\
  \label{ConstraintB}
  \partial_t A_j(0) &= \partial_j A_0(0) + \epsilon_{jk} J^k(0) \in H^{s-1},
\end{align}
where $J_k = 2\im\left(\overline{\phi}D_k\phi\right) = 2\im\left(\overline{\phi}\partial_k\phi\right) + 2A_k\abs{\phi}^2$, hence $J^k(0) \in H^{s-1}$ with norm bounded in terms of the norm of \eqref{WaveData}, as follows from:

\begin{lemma} If $s > 0$, the following estimates hold:
\begin{gather}
  \label{SobA}
  \norm{fg}_{H^{s-1}} \lesssim \norm{f}_{H^{s+1/2}} \norm{g}_{H^{s-1/2}},
  \\
  \label{SobB}
  \norm{fgh}_{H^{s-1}} \lesssim \norm{f}_{H^s} \norm{g}_{H^{s+1/2}} \norm{h}_{H^{s+1/2}}.
\end{gather}
\end{lemma}

\begin{proof}
This follows from the $H^s$ product law in two dimensions (see, e.g., \cite{dfs2010}), which states that, for $s_0,s_1,s_2 \in \R$, the estimate
$$
  \norm{fg}_{H^{-s_0}} \lesssim \norm{f}_{H^{s_1}} \norm{g}_{H^{s_2}}
$$
holds if and only if (i) $s_0 + s_1 + s_2 \ge 1$, (ii) $s_0 + s_1 + s_2 \ge \max(s_0,s_1,s_2)$ and (iii) at most one of (i) and (ii) is an equality. In particular, for $s > 0$ this implies \eqref{SobA}, as well as
$$
  \norm{fg}_{H^{s-1/2}} \lesssim \norm{f}_{H^{s}} \norm{g}_{H^{s-1/2}}
$$
and the latter combined with \eqref{SobA} gives \eqref{SobB}.
\end{proof}

\begin{theorem}\label{Thm3}
The Chern-Simons-Higgs-Lorenz Cauchy problem \eqref{Waves}--\eqref{ConstraintB} is locally well posed if $s > 3/8$, assuming that the potential $V(r)$ is a polynomial of degree $n$, where if $s < 1/2$ we assume $n < 1 + 2/(1-2s)$, whereas if $s \ge 1/2$ there is no restriction on $n$. To be precise, there exists a time $T > 0$, which is a decreasing and continuous function of the initial data norm
$$
  \norm{A(0)}_{H^s} + \norm{\phi(0)}_{H^{s+1/2}} + \norm{\partial_t \phi(0)}_{H^{s-1/2}},
$$
and a solution $(A,\phi)$ of \eqref{Waves} on $(-T,T) \times \R^2$ with the regularity
\begin{gather*}
  A_\mu \in C([-T,T];H^s), \qquad \partial_t A_\mu \in C([-T,T];H^{s-1}),
  \\
  \phi \in C([-T,T];H^{s+1/2}), \qquad \partial_t \phi \in C([-T,T];H^{s-1/2}).
\end{gather*}
The solution is unique in a certain subset of this regularity class. Moreover, the solution depends continuously on the data, and higher regularity persists. In particular, if the data are smooth, then so is the solution.
\end{theorem}

To prove this we iterate in $X^{s,b}$-spaces, so by standard methods we reduce to proving estimates for the right hand sides in \eqref{Waves}. The most difficult terms are the two bilinear ones, for which null structure is needed. The first term on the right hand side of \eqref{WaveA} is already a null form, whereas the first term on the right hand side of \eqref{WaveB} appears also in the Maxwell-Klein-Gordon system in Lorenz gauge, and we showed in \cite{st2010} that it has a null structure. To reveal this structure we transform the variables:
\begin{equation}\label{Splitting}
\begin{gathered}
  A_\mu = A_{\mu,+} + A_{\mu,-}, \qquad \phi = \phi_+ + \phi_-,
  \\
  A_{\mu,\pm} = \frac12 \left( A_\mu \pm i^{-1}\angles{\nabla}^{-1} \partial_t A_\mu \right),
  \qquad
  \phi_{\pm} = \frac12 \left( \phi \pm i^{-1} \angles{\nabla}^{-1}\partial_t \phi\right).
\end{gathered}
\end{equation}
Then \eqref{Waves} transforms to
\begin{subequations}\label{SplitWaves}
\begin{gather}
  \label{SplitWaveA}
  \left(i\partial_t \pm \angles{\nabla} \right) A_{\mu,\pm} = \pm 2^{-1}\angles{\nabla}^{-1} \left( \text{R.H.S. \eqref{WaveA}} \right),
  \\
  \partial^\mu A_\mu = 0,
  \\
  \label{SplitWaveB}
  \left(i\partial_t \pm \angles{\nabla} \right) \phi_\pm = \pm 2^{-1}\angles{\nabla}^{-1} \left( \text{R.H.S. \eqref{WaveB}} \right).
\end{gather}
\end{subequations}

We split the spatial part $\mathbf A=(A_1,A_2)$ of the potential into divergence-free and curl-free parts and a smoother part:
\begin{gather}
  \mathbf A = \mathbf A^{\text{df}} + \mathbf A^{\text{cf}} + (1-\Delta)^{-1} \mathbf A,
  \\
  \label{Adf}
  \mathbf A^{\text{df}} = (R_1R_2A_2 - R_2R_2A_1,R_1R_2A_1-R_1R_1A_2),
  \\
  \label{Acf}
  \mathbf A^{\text{cf}} = (-R_1R_2A_2 - R_1R_1A_1,-R_1R_2A_1-R_2R_2A_2),
\end{gather}
where
$$
  R_j = (1-\Delta)^{-1/2}\partial_j
$$
is bounded on $L^p$, $1 < p < \infty$. Now write
\begin{equation}\label{MainBilinear}
  A_\mu\partial^\mu\phi
  = \left( A_0 \partial_t \phi - \mathbf A^{\text{cf}} \cdot \nabla \phi \right)
  - \mathbf A^{\text{df}} \cdot \nabla \phi - \angles{\nabla}^{-2}\mathbf A \cdot \nabla\phi
  \equiv \mathfrak B_1 - \mathfrak B_2 - \mathfrak B_3,
\end{equation}
where $\mathfrak B_2 = \mathbf A^{\text{df}} \cdot \nabla \phi$ was shown in \cite{km1994} to be a null form:
\begin{equation}\label{P2}
  \mathfrak B_2 = R_2\psi \partial_1\phi - R_1\psi \partial_2\phi,
  \qquad
  \text{where $\psi = R_1A_2-R_2A_1$}.
\end{equation}
In \cite{st2010} we found that $\mathfrak B_1$ also has a null structure: By the Lorenz condition \eqref{LG} we have $R_1A_1+R_2A_2 = \angles{\nabla}^{-1}\partial_t A_0$, hence
\begin{equation}\label{CurlFree}
  A^{\text{cf}}_j = - R_j(R_1A_1+R_2A_2) = - iR_j(A_{0,+}-A_{0,-}),
\end{equation}
where we also used $\partial_t A_0 = i\angles{\nabla} (A_{0,+}-A_{0,-})$. Thus, $\mathfrak B_1 = A_0 \partial_t \phi + A^{\text{cf}}_j \partial^j \phi$ becomes
\begin{equation}\label{P1}
\begin{aligned}
  \mathfrak B_1 &= (A_{0,+} + A_{0,-}) i\angles{\nabla} (\phi_+ - \phi_-)
  - iR_j(A_{0,+}-A_{0,-})\partial^j (\phi_+ + \phi_-)
  \\
  &= i\sum_{\pm_1,\pm_2} \left( A_{0,\pm_1} \angles{\nabla} (\pm_2\phi_{\pm_2})
  - R_j(\pm_1 A_{0,\pm_1})\partial^j\phi_{\pm_2} \right),
\end{aligned}
\end{equation}
where we used $\partial_t \phi = i\angles{\nabla}(\phi_+ - \phi_-)$.

Taking into account \eqref{P2} and \eqref{P1}, we rewrite \eqref{SplitWaves} as
\begin{subequations}\label{ModifiedSplitWaves}
\begin{gather}
  \label{ModifiedSplitWaveA}
  \left(i\partial_t \pm \angles{\nabla} \right) A_{\mu,\pm} = \pm 2^{-1}\angles{\nabla}^{-1} \mathfrak M_\mu(A_+,A_-,\phi_+,\phi_-),
  \\
  \label{LG'}
  \partial^\mu A_\mu = 0,
  \\
  \label{ModifiedSplitWaveB}
  \left(i\partial_t \pm \angles{\nabla} \right) \phi_\pm = \pm 2^{-1}\angles{\nabla}^{-1} \mathfrak N(A_+,A_-,\phi_+,\phi_-),
\end{gather}
\end{subequations}
where
\begin{gather*}
  \mathfrak M_\mu(A_+,A_-,\phi_+,\phi_-)
  =
  - \epsilon_{\mu\nu\rho} \im Q^{\nu\rho}(\partial\overline{\phi},\partial\phi) + 2\epsilon_{\mu\nu\rho} \partial^\nu\left( A^\rho \abs{\phi}^2 \right) + A_\mu,
  \\
  \mathfrak N(A_+,A_-,\phi_+,\phi_-)
  =
  2i (\mathfrak B_1 - \mathfrak B_2 - \mathfrak B_3) + A_\mu A^\mu \phi - \phi V'\left( \abs{\phi}^2 \right) + \phi,
\end{gather*}
with $\mathfrak B_1$ and $\mathfrak B_2$ given by \eqref{P1} and \eqref{P2}, and $\mathfrak B_3 = \angles{\nabla}^{-2}A \cdot\phi$. Here it is understood that $A_\mu = A_{\mu,+} + A_{\mu,-}$, $\phi = \phi_+ + \phi_-$, $\partial_t A_\mu = i\angles{\nabla}(A_{\mu,+} - A_{\mu,-})$, and $\partial_t \phi = i\angles{\nabla}(\phi_+ - \phi_-)$.

The initial data are
\begin{equation}\label{SplitData}
\begin{gathered}
  A_{\mu,\pm}(0) = \frac12 \left( A_\mu(0) \pm i^{-1}\angles{\nabla}^{-1} \partial_t A_\mu(0) \right) \in H^s,
  \\
  \phi_{\pm}(0) = \frac12 \left( \phi(0) \pm i^{-1} \angles{\nabla}^{-1}\partial_t \phi(0) \right) \in H^{s+1/2}.
\end{gathered}
\end{equation}

The systems \eqref{ModifiedSplitWaves} and \eqref{Waves} are equivalent via the transformation \eqref{Splitting}, so it suffices to solve \eqref{ModifiedSplitWaves}. The Lorenz condition \eqref{LG'} reduces to an initial constraint, since if $(A_+,A_-,\phi_+,\phi_-)$ is a solution of \eqref{ModifiedSplitWaves}, then setting $A = A_+ + A_-$ and $\phi = \phi_+ + \phi_-$ we have $(\square + 1)A_\mu = \mathcal M_\mu$, so \eqref{WaveA} is satisfied, i.e., $\square A_\mu = -\epsilon_{\mu\nu\rho} \partial^\nu J^\rho$. Thus, $u=\partial^\mu A_\mu$ satisfies $\square u = 0$, and $u(0) = \partial_t u(0) = 0$ by \eqref{LorenzConstraint} and \eqref{ConstraintB}.

We prove local well-posedness of \eqref{ModifiedSplitWaves} by iterating in the $X^{s,b}$-spaces adapted to the operators $i\partial_t \pm \angles{\nabla}$, so by standard arguments (see, e.g., \cite{st2010} for more details) the proof of Theorem \ref{Thm3} reduces to proving, for some $b,b' \in (1/2,1)$, $m \ge 2$, and $\varepsilon > 0$, the estimates
\begin{gather}
  \label{NonlinearA}
  \norm{\mathfrak M(A_+,A_-,\phi_+,\phi_-)}_{X_\pm^{s-1,b-1+\varepsilon}} \lesssim B+B^m,
  \\
  \label{NonlinearB}
  \norm{\mathfrak N(A_+,A_-,\phi_+,\phi_-)}_{X_\pm^{s-1/2,b'-1+\varepsilon}} \lesssim B+B^m,
\end{gather}
where
$$
  B = \norm{\phi_+}_{X_+^{s+1/2,b'}} + \norm{\phi_-}_{X_-^{s+1/2,b'}}
  + \sum_{\mu=0}^2 \left( \norm{A_{\mu,+}}_{X_+^{s,b}} + \norm{A_{\mu,-}}_{X_-^{s,b}} \right)
$$
and
$$
  \norm{u}_{X_\pm^{s,b}} = \norm{ \angles{\xi}^s \angles{-\tau\pm\abs{\xi}}^b \widehat u(\tau,\xi)}_{L^2_{\tau,\xi}}.
$$
Here $\widehat u(\tau,\xi)$ is the space-time Fourier transform of $u(t,x)$. Note that $\angles{-\tau\pm\abs{\xi}}$ is comparable to $\angles{-\tau\pm\angles{\xi}}$.

We also need the wave-Sobolev norms
$$
  \norm{u}_{H^{s,b}} = \norm{ \angles{\xi}^s \angles{\abs{\tau}-\abs{\xi}}^b \widehat u(\tau,\xi)}_{L^2_{\tau,\xi}}.
$$
Frequent use will be made of the fact that $\norm{u}_{X_\pm^{a,\alpha}} \le \norm{u}_{H^{a,\alpha}}$ if $\alpha \le 0$, and that the reverse inequality holds if $\alpha \ge 0$. In particular, it suffices to prove \eqref{NonlinearA} and \eqref{NonlinearB} with the $X$-norms on the left hand sides replaced by the corresponding $H$-norms. 

\subsection{Proof of \eqref{NonlinearA} for $\mathfrak M_{\mu,1} = - \epsilon_{\mu\nu\rho} \im Q^{\nu\rho}(\partial\overline{\phi},\partial\phi)$} We shall prove that
\begin{equation}\label{M1estimate}
  \norm{Q^{\nu\rho}(\partial\overline{\phi},\partial\phi)}_{H^{s-1,b-1+\varepsilon}}
  \lesssim \norm{\phi_+}_{X_+^{s+1/2,b'}}^2 + \norm{\phi_-}_{X_-^{s+1/2,b'}}^2
\end{equation}
holds if
\begin{equation}\label{Conditions1}
  \frac12 < b,b' < 1, \qquad s > \max\left(b-\frac12, \frac14, \frac16+\frac{b}{3}, \frac{b}{2} \right),
\end{equation}
and $\varepsilon > 0$ is sufficiently small.

Observe that
\begin{gather*}
  Q_{jk}(\partial\overline\phi,\partial\phi) = \sum_{\pm_1,\pm_2} \left( \partial_j\overline{\phi_{\pm_1}} \partial_k\phi_{\pm_2} - \partial_k\overline{\phi_{\pm_1}} \partial_j\phi_{\pm_2} \right),
  \\
  Q_{0j}(\partial\overline\phi,\partial\phi) = \sum_{\pm_1,\pm_2} \left( -i\angles{\nabla}\left(\pm_1\overline{\phi_{\pm_1}}\right) \partial_j\phi_{\pm_2} - \partial_j\overline{\phi_{\pm_1}} i\angles{\nabla} \left(\pm_2\phi_{\pm_2}\right) \right),
\end{gather*}
where we used $\partial_t \phi = i\angles{\nabla}(\phi_+ - \phi_-)$. Since $\norm{\overline{u}}_{X_\pm^{s,b}} = \norm{u}_{X_{\mp}^{s,b}}$, it suffices to show
\begin{gather}
  \label{NullFormA}
  \norm{\partial_j u \partial_k v - \partial_k u \partial_j v}_{H^{s-1,b-1+\varepsilon}} \lesssim \norm{u}_{X_{\pm_1}^{s+1/2,b'}} \norm{v}_{X_{\pm_2}^{s+1/2,b'}},
  \\
  \label{NullFormB}
  \norm{\partial_j (\pm_1 u) \angles{\nabla}v - \angles{\nabla}u \partial_j (\pm_2 v)}_{H^{s-1,b-1+\varepsilon}} \lesssim \norm{u}_{X_{\pm_1}^{s+1/2,b'}} \norm{v}_{X_{\pm_2}^{s+1/2,b'}}.
\end{gather}
The left hand sides are bounded by $\norm{I(\tau,\xi)}_{L^2_{\tau,\xi}}$, where
\begin{equation}\label{Idef}
  I(\tau,\xi)
  =
  \int_{\R^{1+2}} \frac{\sigma\left(\pm_1\eta,\pm_2(\xi-\eta)\right)}{\angles{\xi}^{1-s} \angles{\abs{\tau}-\abs{\xi}}^{1-b-\varepsilon}}
  \abs{ \widehat{u}(\lambda,\eta) } \abs{ \widehat{v}(\tau-\lambda,\xi-\eta) } \, d\lambda \, d\eta
\end{equation}
and $\sigma$ is either
$$
  \sigma(\eta,\zeta) = \abs{\eta \times \zeta} \le \abs{\eta}\abs{\zeta} \theta(\eta,\zeta)
$$
or
$$
  \sigma(\eta,\zeta) = \abs{\angles{\eta}\zeta_j - \eta_j\angles{\zeta}}
   \le \abs{\eta}\abs{\zeta} \theta(\eta,\zeta)
  + \frac{\abs{\zeta}}{\angles{\eta}}
  + \frac{\abs{\eta}}{\angles{\zeta}}.
$$
Here $\theta(\eta,\zeta)$ denotes the angle between nonzero vectors $\eta,\zeta \in \R^2$.

We now use the following estimate from \cite{st2010}:

\begin{lemma}\label{AngleLemma}
For all signs $(\pm_1,\pm_2)$, all $\lambda,\mu \in \R$, and all nonzero $\eta,\zeta \in \R^2$,
$$
  \theta(\pm_1\eta,\pm_2\zeta)
  \lesssim
  \left(
  \frac{\angles{\abs{\lambda+\mu}-\abs{\eta+\zeta}} + \angles{-\lambda\pm_1\abs{\eta}}
  + \angles{-\mu\pm_2\abs{\zeta}} }
  {\min(\angles{\eta},\angles{\zeta})}
  \right)^{1/2}.
$$
\end{lemma}

Thus we reduce \eqref{NullFormA} and \eqref{NullFormB} to (recalling $\norm{u}_{H^{a,\alpha}} \le \norm{u}_{X_\pm^{a,\alpha}}$, $\alpha \ge 0$)
\begin{align*}
  \norm{uv}_{H^{s-1,b-1/2+\varepsilon}} &\lesssim \norm{u}_{H^{s,b'}} \norm{v}_{H^{s-1/2,b'}},
  \\
  \norm{uv}_{H^{s-1,b-1+\varepsilon}} &\lesssim \norm{u}_{H^{s,b'-1/2}} \norm{v}_{H^{s-1/2,b'}},
  \\
  \norm{uv}_{H^{s-1,b-1+\varepsilon}} &\lesssim \norm{u}_{H^{s,b'}} \norm{v}_{H^{s-1/2,b'-1/2}},
  \\
  \norm{uv}_{H^{s-1,0}} &\lesssim \norm{u}_{H^{s+3/2,b'}} \norm{v}_{H^{s-1/2,b'}}.
\end{align*}
Assuming \eqref{Conditions1}, all these estimates hold by the following product law.

\begin{theorem}\label{Thm4} \emph{(D'Ancona, Foschi and Selberg \cite{dfs2010}.)}
Let $s_0,s_1,s_2,b_0,b_1,b_2 \in \R$. The product estimate in $1+2$ dimensions,
$$
  \norm{uv}_{H^{-s_0,-b_0}} \le C \norm{u}_{H^{s_1,b_1}}\norm{v}_{H^{s_2,b_2}},
$$
holds for all $u,v \in \mathcal S(\R^{1+2})$ if the following conditions are satisfied:
\begin{align*}
  & b_0 + b_1 + b_2 > \frac12,
  \\
  & b_0 + b_1 \ge 0,
  \\
  & b_0 + b_2 \ge 0,
  \\
  & b_1 + b_2 \ge 0,
  \\
  & s_0 + s_1 + s_2 > \frac32 - (b_0 + b_1 + b_2),
  \\
  & s_0 + s_1 + s_2 > 1 - \min(b_0 + b_1, b_0 + b_2, b_1 + b_2),
  \\
  & s_0 + s_1 + s_2 > \frac12 - \min(b_0,b_1,b_2),
  \\
  & s_0 + s_1 + s_2 > \frac34,
  \\
  & (s_0 + b_0) + 2s_1+ 2s_2 > 1,
  \\
  & 2s_0 + (s_1 + b_1) + 2s_2 > 1,
  \\
  & 2s_0 + 2s_1 + (s_2 + b_2) > 1,
  \\
  & s_1 + s_2 \ge \max(0,-b_0),
  \\
  & s_0 + s_2 \ge \max(0,-b_1),
  \\
  & s_0 + s_1 \ge \max(0,-b_2).
\end{align*}
\end{theorem}

We remark that this product law is optimal up to endpoint cases. A more precise statement, including many endpoint cases, can be found in \cite{dfs2010}.

\subsection{Proof of \eqref{NonlinearA} for $\mathfrak M_{\mu,2} = 2\epsilon_{\mu\nu\rho} \partial^\nu\left( A^\rho \abs{\phi}^2 \right)$} By Leibniz's rule and using $\partial_t A_j = i\angles{\nabla}(A_{j,+}-A_{j,-})$ and $\partial_t \phi = i\angles{\nabla}(\phi_+ - \phi_-)$, we reduce to
\begin{align}
  \label{TrilinearA}
  \norm{uvw}_{H^{s-1,b-1+\varepsilon}} &\lesssim \norm{u}_{H^{s-1,b}} \norm{v}_{H^{s+1/2,b'}} \norm{w}_{H^{s+1/2,b'}},
  \\
  \label{TrilinearB}
  \norm{uvw}_{H^{s-1,b-1+\varepsilon}} &\lesssim \norm{u}_{H^{s,b}} \norm{v}_{H^{s+1/2,b'}} \norm{w}_{H^{s-1/2,b'}}.
\end{align}

But \eqref{TrilinearA} follows by two applications of Theorem \ref{Thm4}:
\begin{align*}
  \norm{uvw}_{H^{s-1,b-1+\varepsilon}} &\lesssim \norm{u}_{H^{s-1,b}} \norm{vw}_{H^{s+1/2,b-1/2}}
  \\
  &\lesssim
  \norm{u}_{H^{s-1,b}} \norm{v}_{H^{s+1/2,b'}} \norm{w}_{H^{s-1/2,b'}},
\end{align*}
provided that
\begin{equation}\label{Conditions1b}
  \frac12 < b,b' < 1, \qquad s > \max\left(1-b, b-\frac12, \frac14, \frac{b}{3} \right),
\end{equation}
and $\varepsilon > 0$ is sufficiently small. Moreover, assuming only $s > \max(b-1/2,1/4)$,
\begin{align*}
  \norm{uvw}_{H^{s-1,b-1+\varepsilon}} &\lesssim \norm{uv}_{H^{s,0}} \norm{w}_{H^{s-1/2,b'}}
  \\
  &\lesssim
  \norm{u}_{H^{s,b}} \norm{v}_{H^{s+1/2,b'}} \norm{w}_{H^{s-1/2,b'}},
\end{align*}
so \eqref{Conditions1b} is more than sufficient for \eqref{TrilinearB} to hold also.

\subsection{Proof of \eqref{NonlinearA} for $\mathfrak M_{\mu,3} = A_\mu$} Trivially,
$$
  \norm{A_\mu}_{H^{s-1,b-1+\varepsilon}} \le \norm{A_\mu}_{H^{s,b}} \le \sum_{\pm} \norm{A_{\mu,\pm}}_{X_\pm^{s,b}}.
$$

\subsection{Proof of \eqref{NonlinearB} for $\mathfrak N_1 = 2i(\mathfrak B_1-\mathfrak B_2-\mathfrak B_3)$} We need
\begin{equation}\label{Pestimate}
  \norm{\mathfrak B_j}_{H^{s-1/2,b'-1+\varepsilon}} \lesssim \left(\sum_\mu\sum_{\pm}\norm{A_{\mu,\pm}}_{X_\pm^{s,b}}\right) \left(\sum_{\pm}\norm{\phi_\pm}_{X_\pm^{s+1/2,b'}}\right)
\end{equation}
for $j=1,2,3$. For $\mathfrak B_2$ given by \eqref{P2} we reduce to
$$
  \norm{R_1 u \partial_2 v - R_2 u \partial_1 v}_{H^{s-1/2,b'-1+\varepsilon}} \lesssim \norm{u}_{X_{\pm_1}^{s,b}} \norm{v}_{X_{\pm_2}^{s+1/2,b'}},
$$
and proceeding as we did for \eqref{NullFormA}, we further reduce to
\begin{equation}\label{Products2}
\begin{aligned}
  \norm{uv}_{H^{s-1/2,b'-1/2+\varepsilon}} &\lesssim \norm{u}_{H^{s,b}} \norm{v}_{H^{s,b'}} ,
  \\
  \norm{uv}_{H^{s-1/2,b'-1/2+\varepsilon}} &\lesssim \norm{u}_{H^{s+1/2,b}} \norm{v}_{H^{s-1/2,b'}},
  \\
  \norm{uv}_{H^{s-1/2,b'-1+\varepsilon}} &\lesssim \norm{u}_{H^{s,b}} \norm{v}_{H^{s,b'-1/2}},
  \\
  \norm{uv}_{H^{s-1/2,b'-1+\varepsilon}} &\lesssim \norm{u}_{H^{s+1/2,b}} \norm{v}_{H^{s-1/2,b'-1/2}},
  \\
  \norm{uv}_{H^{s-1/2,b'-1+\varepsilon}} &\lesssim \norm{u}_{H^{s,b-1/2}} \norm{v}_{H^{s,b'}},
  \\
  \norm{uv}_{H^{s-1/2,b'-1+\varepsilon}} &\lesssim \norm{u}_{H^{s+1/2,b-1/2}} \norm{v}_{H^{s-1/2,b'}},
\end{aligned}
\end{equation}
all of which hold by Theorem \ref{Thm4} provided that
\begin{equation}\label{Conditions2}
  \frac12 < b,b' < 1, \qquad s > \max\left( b'-\frac12, \frac14, \frac{b'}{3}, 1-b \right),
\end{equation}
and $\varepsilon > 0$ is sufficiently small. Thus we have \eqref{Pestimate} for $\mathfrak B_2$.

For $\mathfrak B_1$ given by \eqref{P1}, the estimate \eqref{Pestimate} reduces to
$$
  \norm{u \angles{\nabla} v
  - R_j(\pm_1 u) \partial^j(\pm_2 v)}_{H^{s-1/2,b'-1+\varepsilon}}
  \lesssim \norm{u}_{X_{\pm_1}^{s,b}} \norm{v}_{X_{\pm_2}^{s+1/2,b'}}.
$$
The left hand side is bounded by $\norm{I(\tau,\xi)}_{L^2_{\tau,\xi}}$, where $I$ is given by \eqref{Idef} with
\begin{equation}\label{Symbol}
  \sigma(\eta,\zeta) = \abs{ \angles{\zeta} - \frac{\eta \cdot \zeta}{\angles{\eta}} }
  \lesssim
  \angles{\zeta} \theta^2(\eta,\zeta) + \angles{\zeta} \left( \frac{1}{\angles{\eta}^2} + \frac{1}{\angles{\zeta}^2} \right),
\end{equation}
so we reduce to \eqref{Products2} and two additional estimates:
\begin{align}
  \label{AdditionalA}
  \norm{uv}_{H^{s-1/2,0}} &\lesssim \norm{u}_{H^{s+2,b}} \norm{v}_{H^{s-1/2,b'}},
  \\
  \norm{uv}_{H^{s-1/2,0}} &\lesssim \norm{u}_{H^{s,b}} \norm{v}_{H^{s+3/2,b'}},
\end{align}
which hold by Theorem \ref{Thm4} if $b,b' > 1/4$ and $s \ge - 3/4$, so \eqref{Conditions2} is more than sufficient.

Finally, the estimate for $\mathfrak B_3 = \angles{\nabla}^{-2} \mathbf A \cdot \nabla \phi$ reduces to \eqref{AdditionalA}.

\subsection{Proof of \eqref{NonlinearB} for $\mathfrak N_2 = A_\mu A^\mu \phi$} For this we need
$$
  \norm{uvw}_{H^{s-1/2,b'-1+\varepsilon}} \lesssim \norm{u}_{H^{s,b}} \norm{v}_{H^{s,b}} \norm{w}_{H^{s+1/2,b'}}.
$$
But two applications of Theorem \ref{Thm4} give
\begin{align*}
  \norm{uvw}_{H^{s-1/2,b'-1+\varepsilon}} &\lesssim \norm{u}_{H^{s,b}} \norm{vw}_{H^{s,0}}
  \\
  &\lesssim
  \norm{u}_{H^{s,b}} \norm{v}_{H^{s,b}} \norm{w}_{H^{s+1/2,b'}},
\end{align*}
assuming \eqref{Conditions2} holds.

\subsection{Proof of \eqref{NonlinearB} for $\mathfrak N_3 = - \phi V'\left( \abs{\phi}^2 \right)$} If $V(r)$ is a polynomial, we need
\begin{equation}\label{Nproduct}
  \norm{u_1 \cdots u_N}_{H^{s-1/2,b'-1+\varepsilon}} \lesssim \prod_{j=1}^N \norm{u_j}_{H^{s+1/2,b'}}.
\end{equation}
Writing
$$
  \sigma = s-\frac12+\theta, \qquad \beta = b'-1+\varepsilon+\theta, \qquad 0 \le \theta \le 1-\delta-\varepsilon,
$$
the estimate
\begin{equation}\label{StepEstimate}
  \norm{uv}_{H^{\sigma,\beta}} \lesssim \norm{u}_{H^{s+1/2,b'}} \norm{v}_{H^{\sigma+\delta,\beta+\delta}}
\end{equation}
holds by Theorem \ref{Thm4} if
\begin{equation}\label{Conditions3}
  b' > 1/2, \qquad \delta > 0, \qquad s > \max\left( b'-\frac12, b'-2\delta, \frac14-\delta \right).
\end{equation}
Applying \eqref{StepEstimate} inductively, we get \eqref{Nproduct} provided $(N-1)\delta < 1$. 

\subsection{Proof of \eqref{NonlinearB} for $\mathfrak N_4 = \phi$} Trivially,
$$
  \norm{\phi}_{H^{s-1/2,b'-1+\varepsilon}} \le \norm{\phi}_{H^{s+1/2,b'}} \le \sum_{\pm} \norm{\phi_\pm}_{X_\pm^{s+1/2,b'}}.
$$

\subsection{Conclusion of the proof: The choice of $s,b,b'$} We have proved that \eqref{NonlinearA} and \eqref{NonlinearB} hold under the conditions \eqref{Conditions1}, \eqref{Conditions1b}, \eqref{Conditions2} and \eqref{Conditions3}, where $\varepsilon > 0$ is arbitrarily small. The optimal choice for $b'$ is obviously $b'=1/2+\varepsilon$, and then the condition $s > b'- 2\delta$ from \eqref{Conditions3} is satisfied if we set
$$
  \delta =
  \begin{cases}
  \varepsilon + \frac{1/2-s}{2} &\text{if $s < 1/2$},
  \\
  \varepsilon &\text{if $s \ge 1/2$}.
  \end{cases}
$$

The condition $(N-1)\delta < 1$ for \eqref{Nproduct} to hold then becomes $N < 1 + 4/(1-2s)$ if $s < 1/2$, whereas $N$ is unrestricted if $s \ge 1/2$. For the degree $n$ of the polynomial $V(r)$, this gives $n < 1 + 2/(1-2s)$ if $s < 1/2$, and no restriction if $s \ge 1/2$.

We are left with the conditions $b \in (1/2,1)$ and $s > 1/4,b-1/2,1/6+b/3,1-b$, and optimizing this leads to the conditions $b=5/8$ and $s > 3/8$.

\section{Local well-posedness for finite-energy data}\label{EnergyLWP}

Here we prove Theorem \ref{Thm1}, or rather the following equivalent statement:

\begin{theorem}\label{Thm4}
If $s=1/2$, the analogue of Theorem \ref{Thm3} holds with the data space $H^{1/2} \times H^{-1/2}$ for $(A_\mu,\partial_t A_\mu)$ replaced by its homogeneous counterpart $\dot H^{1/2} \times \dot H^{-1/2}$, and we allow any potential $V \in C^\infty(\R_+;\R)$ such that $V(0) = 0$ and all derivatives of $V$ have polynomial growth.
\end{theorem}

We remark that for existence one only needs that $V'(r)$ has polynomial growth, but to get persistence of higher regularity one must take additional derivatives of the equations, hence the same assumption is required on all higher derivatives.

The proof follows closely that of Theorem \ref{Thm3}. We do not add $A_\mu$ to each side of the wave equation for $A_\mu$, but use
\begin{gather*}
  \square A_\mu = - \epsilon_{\mu\nu\rho} \im Q^{\nu\rho}(\partial\overline{\phi},\partial\phi) + 2\epsilon_{\mu\nu\rho} \partial^\nu\left( A^\rho \abs{\phi}^2 \right),
  \\
  \partial^\mu A_\mu = 0,
  \\
  (\square+1) \phi = 2i A_\mu \partial^\mu\phi + A_\mu A^\mu \phi - \phi V'\left( \abs{\phi}^2 \right) + \phi,
\end{gather*}
with data
\begin{equation}\label{EnergyData}
  A_\mu(0) \in \dot H^{1/2},
  \qquad (\phi,\partial_t\phi)(0) \in H^1 \times L^2,
\end{equation}
whereas the data for $\partial_t A_\mu$ are given by the constraints \eqref{LorenzConstraint} and \eqref{ConstraintB}, hence they belong to $\dot H^{-1/2}$, recalling from Section \ref{Introduction} that $J^k(0) \in \dot H^{-1/2}$, with norm bounded in terms of the norm of \eqref{EnergyData}.

We modify \eqref{Splitting} by setting $2A_{\mu,\pm} = A_\mu \pm i^{-1}\abs{\nabla}^{-1} \partial_t A_\mu$. The splitting of $\mathbf A=(A_1,A_2)$ into divergence-free and curl-free parts now reads $\mathbf A = \mathbf A^{\text{df}} + \mathbf A^{\text{cf}}$, where $\mathbf A^{\text{df}}$ and $\mathbf A^{\text{cf}}$ are still given by \eqref{Adf} and \eqref{Acf}, but now $R_j$ is the Riesz transform
$$
  R_j = (-\Delta)^{-1/2}\partial_j,
$$
bounded on $L^p$, $1 < p < \infty$. Then \eqref{MainBilinear} remains valid, but without the term $\mathfrak B_3$, and with $\mathfrak B_1$ and $\mathfrak B_2$ given by \eqref{P1} and \eqref{P2}. Thus we obtain the system
\begin{gather*}
  \left(i\partial_t \pm \abs{\nabla} \right) A_{\mu,\pm} = \pm 2^{-1}\abs{\nabla}^{-1} \mathfrak M_\mu(A_+,A_-,\phi_+,\phi_-),
  \\
  \left(i\partial_t \pm \angles{\nabla} \right) \phi_\pm = \pm 2^{-1}\angles{\nabla}^{-1} \mathfrak N(A_+,A_-,\phi_+,\phi_-),
\end{gather*}
where
\begin{gather*}
  \mathfrak M_\mu(A_+,A_-,\phi_+,\phi_-)
  =
  - \epsilon_{\mu\nu\rho} \im Q^{\nu\rho}(\partial\overline{\phi},\partial\phi) + 2\epsilon_{\mu\nu\rho} \partial^\nu\left( A^\rho \abs{\phi}^2 \right),
  \\
  \mathfrak N(A_+,A_-,\phi_+,\phi_-)
  =
  2i (\mathfrak B_1 - \mathfrak B_2) + A_\mu A^\mu \phi - \phi V'\left( \abs{\phi}^2 \right) + \phi,
\end{gather*}
with $\mathfrak B_1$ and $\mathfrak B_2$ given by \eqref{P1} and \eqref{P2}. Here it is understood that $A_\mu = A_{\mu,+} + A_{\mu,-}$, $\phi = \phi_+ + \phi_-$, $\partial_t A_\mu = i\abs{\nabla}(A_{\mu,+} - A_{\mu,-})$, and $\partial_t \phi = i\angles{\nabla}(\phi_+ - \phi_-)$.
 The initial data are
\begin{gather*}
  A_{\mu,\pm}(0) = \frac12 \left( A_\mu \pm(0) i^{-1}\abs{\nabla}^{-1} \partial_t A_\mu(0) \right) \in \dot H^{1/2},
  \\
  \phi_{\pm}(0) = \frac12 \left( \phi(0) \pm i^{-1} \angles{\nabla}^{-1}\partial_t \phi(0) \right) \in H^1.
\end{gather*}
Local well-posedness reduces to proving
\begin{gather}
  \label{HomNonlinearA}
  \norm{\abs{\nabla}^{-1/2} \mathfrak M}_{H^{0,b-1+\varepsilon}} \lesssim B+B^m,
  \\
  \label{HomNonlinearB}
  \norm{\mathfrak N}_{H^{0,b'-1+\varepsilon}} \lesssim B+B^m,
\end{gather}
where
$$
  B = \sum_\pm \norm{\phi_\pm}_{X_\pm^{1,b'}}
  + \sum_{\mu=0}^2 \sum_\pm \norm{\abs{\nabla}^{1/2} A_{\mu,\pm}}_{X_\pm^{0,b}}.
$$

Let $P_{\abs{\xi}<1}$ and $P_{\abs{\xi} \ge 1}$ be the multipliers with symbols $\chi_{\abs{\xi}<1}$ and $\chi_{\abs{\xi} \ge 1}$, which we use to split $f(x)$ into low- and high-frequency parts: $f = P_{\abs{\xi}<1}f + P_{\abs{\xi} \ge 1}f$.

Note the estimates
\begin{gather}
  \label{LF1}
  \norm{\abs{\nabla}^{-1/2} P_{\abs{\xi}<1} f}_{L^2}
  \simeq \norm{\frac{\chi_{\abs{\xi}<1}}{\abs{\xi}^{1/2}} \widehat f(\xi)}_{L^2_\xi} \le \left( \int_{\abs{\xi} < 1} \frac{d\xi}{\abs{\xi}} \right)^{1/2} \bignorm{\widehat f\,}_{L^\infty} \lesssim \norm{f}_{L^1},
  \\
  \label{LF2}
  \norm{P_{\abs{\xi}<1} f}_{L^p} \lesssim \norm{\abs{\nabla}^{1/2} \angles{\nabla}^{1/2+\varepsilon} P_{\abs{\xi}<1} f}_{L^2} \lesssim \norm{\abs{\nabla}^{1/2} f}_{L^2}, \quad 4 \le p \le \infty.
\end{gather}

\subsection{Proof of \eqref{HomNonlinearA} for $\mathfrak M_{\mu,1} = -\epsilon_{\mu\nu\rho} \im Q^{\nu\rho}(\partial\overline{\phi},\partial\phi)$} Splitting into low and high frequencies and applying \eqref{LF1} we get
$$
  \norm{\abs{\nabla}^{-1/2} \mathfrak M_{\mu,1}}_{H^{0,b-1+\varepsilon}}
  \le
  \norm{\partial\phi}_{L^4_t(L^2_x)}^2
  +
  \norm{\angles{\nabla}^{-1/2} \mathfrak M_{\mu,1}}_{H^{0,b-1+\varepsilon}}
  \lesssim
  \sum_{\pm} \norm{\phi_\pm}_{X_\pm^{1,b'}}^2,
$$
where we used \eqref{M1estimate} and $\norm{\partial\phi}_{L^4_t(L^2_x)} \lesssim \norm{\partial\phi}_{H^{0,b'}} \lesssim \sum_{\pm} \norm{\phi_\pm}_{X_\pm^{1,b'}}$.

\subsection{Proof of \eqref{HomNonlinearA} for $\mathfrak M_{\mu,2} = 2\epsilon_{\mu\nu\rho} \partial^\nu\left( A^\rho \abs{\phi}^2 \right)$} By Leibniz's rule and (if $\nu = 0$) the fact that $\partial_t A_j = i\abs{\nabla}(A_{j,+}-A_{j,-})$ and $\partial_t \phi = i\angles{\nabla}(\phi_+ - \phi_-)$, we reduce to
\begin{align}
  \label{HomTrilinearA}
  \norm{\abs{\nabla}^{-1/2}\left( \abs{\nabla} u vw \right)}_{L^2_{t,x}}
  &\lesssim
  \norm{\abs{\nabla}^{1/2}u}_{H^{0,b}}
  \norm{v}_{H^{1,b'}}
  \norm{w}_{H^{1,b'}},
  \\
  \label{HomTrilinearB}
  \norm{\abs{\nabla}^{-1/2}\left( u \angles{\nabla} v w \right)}_{L^2_{t,x}}
  &\lesssim
  \norm{\abs{\nabla}^{1/2}u}_{H^{0,b}}
  \norm{v}_{H^{1,b'}}
  \norm{w}_{H^{1,b'}}.
\end{align}

First consider the low-frequency case where we replace $\abs{\nabla}^{-1/2}$ on the left hand side by $\abs{\nabla}^{-1/2} P_{\abs{\xi} < 1}$. Then \eqref{HomTrilinearA} reduces to \eqref{HomTrilinearB}, since by the triangle inequality in Fourier space, and assuming as we may that $\widehat u, \widehat v, \widehat w \ge 0$, we have
$$
  \mathcal F \left( \abs{\nabla} u vw \right)(\tau,\xi) \le \abs{\xi} \mathcal F \left( u vw \right)(\tau,\xi) + \mathcal F \left( u \angles{\nabla}v w \right)(\tau,\xi) + \mathcal F \left( u v \angles{\nabla} w \right)(\tau,\xi).
$$
But \eqref{HomTrilinearB} is easily proved by applying \eqref{LF1}:
\begin{align*}
  &\norm{\abs{\nabla}^{-1/2}P_{\abs{\xi} < 1}\left( u \angles{\nabla} v w \right)}_{L^2_{t,x}}
  \lesssim
  \norm{u \angles{\nabla} v w}_{L^2_t(L^1_x)}
  \\
  &\quad\le \norm{u}_{L^4_{t,x}} \norm{\angles{\nabla} v}_{L^\infty_t(L^2_x)} \norm{w}_{L^4_{t,x}}
  \lesssim
  \norm{\abs{\nabla}^{1/2}u}_{H^{0,b}}
  \norm{v}_{H^{1,b'}}
  \norm{w}_{H^{1,b'}},
\end{align*}
where we used the Sobolev embedding $\dot H^{1/2} \subset L^4$ on $\R^2$.

Now consider the high-frequency case where we replace $\abs{\nabla}^{-1/2}$ on the left hand side by $\abs{\nabla}^{-1/2} P_{\abs{\xi} \ge 1}$. This case obviously reduces to
\begin{align}
  \label{HomTrilinearC}
  \norm{\angles{\nabla}^{-1/2}\left( \abs{\nabla} u vw \right)}_{L^2_{t,x}}
  &\lesssim
  \norm{\abs{\nabla}^{1/2}u}_{H^{0,b}}
  \norm{v}_{H^{1,b'}}
  \norm{w}_{H^{1,b'}},
  \\
  \label{HomTrilinearD}
  \norm{\angles{\nabla}^{-1/2}\left( u \angles{\nabla} v w \right)}_{L^2_{t,x}}
  &\lesssim
  \norm{\abs{\nabla}^{1/2}u}_{H^{0,b}}
  \norm{v}_{H^{1,b'}}
  \norm{w}_{H^{1,b'}},
\end{align}
and if $u$ is replaced by $P_{\abs{\xi} \ge 1} u$, these in turn reduce to \eqref{TrilinearA} and \eqref{TrilinearB}. On the other hand, if $u$ is replaced by $P_{\abs{\xi} < 1} u$, \eqref{HomTrilinearC} and \eqref{HomTrilinearD} both follow from
\begin{align*}
  &\norm{\angles{\nabla}^{-1/2}\left( P_{\abs{\xi} < 1} u vw \right)}_{L^2_{t,x}}
  \lesssim
  \norm{P_{\abs{\xi} < 1} u vw}_{L^2_t(L^{4/3}_x)}
  \\
  &\quad\le
  \norm{P_{\abs{\xi} < 1} u}_{L^6_t(L^{\infty}_x)}
  \norm{v}_{L^6_t(L^{4}_x)}
  \norm{w}_{L^6_t(L^{2}_x)}
  \lesssim
  \norm{\abs{\nabla}^{1/2}u}_{H^{0,b}}
  \norm{v}_{H^{1,b'}}
  \norm{w}_{H^{0,b'}},
\end{align*}
where we used \eqref{LF2}.

\subsection{Proof of \eqref{HomNonlinearB} for $\mathcal N_1 = 2i(\mathfrak B_1 - \mathfrak B_2)$ and $\mathcal N_2 = A_\mu A^\mu \phi$} For the high-frequency part $P_{\abs{\xi} \ge 1}A_\mu$ the estimates in Section \ref{LWP} apply, but note that since the definition of $R_j$ has changed, the estimate \eqref{Symbol} is replaced by
$$
  \sigma(\eta,\zeta) = \abs{ \angles{\zeta} - \frac{\eta \cdot \zeta}{\abs{\eta}} }
  \lesssim
  \abs{\zeta} \theta^2(\eta,\zeta) + \frac{1}{\angles{\zeta}}.
$$
For the low-frequency part we reduce to
\begin{gather*}
  \norm{P_{\abs{\xi} < 1}uv}_{L^2_{t,x}}
  \le 2\norm{P_{\abs{\xi} < 1}u}_{L^\infty_{t,x}} \norm{v}_{L^2_{t,x}}
  \lesssim \norm{\abs{\nabla}^{1/2}u}_{H^{0,b}} \norm{v}_{H^{0,b'}},
  \\
  \norm{(P_{\abs{\xi} < 1}u)^2w}_{L^2_{t,x}}
  \le \norm{P_{\abs{\xi} < 1}u}_{L^\infty_{t,x}}^2 \norm{w}_{L^2_{t,x}}
  \lesssim \norm{\abs{\nabla}^{1/2}u}_{H^{0,b}}^2 \norm{w}_{H^{0,b'}},
\end{gather*}
where we used \eqref{LF2}.

\subsection{Proof of \eqref{HomNonlinearB} for $- \phi V'\left( \abs{\phi}^2 \right)$}
Assuming $\abs{V'(r)} \lesssim 1+r^M$ for any $M \ge 1$,
\begin{align*}
  &\norm{\phi V'\left( \abs{\phi}^2 \right)}_{H^{0,b'-1+\varepsilon}}
  \le
  \norm{\phi V'\left( \abs{\phi}^2 \right)}_{L^2_{t,x}}
  \\
  &\qquad\le
  \norm{\phi}_{L^2_{t,x}} + \norm{\phi}_{L^{4M+2}_{t,x}}^{2M+1}
  \lesssim
  \norm{\phi}_{L^2_{t,x}} + \norm{\phi}_{L_t^{4M+2} H^1}^{2M+1}
  \lesssim
  \norm{\phi}_{H^{1,b'}} + \norm{\phi}_{H^{1,b'}}^{2M+1},
\end{align*}
where we used the Sobolev estimate $\norm{f}_{L^p} \lesssim \norm{f}_{H^1}$, $2 \le p < \infty$.

\end{document}